\documentclass{amsart}[10pt]

\usepackage{amsfonts,amssymb,amsmath,amsthm}

\usepackage{tcolorbox} 
\newtcolorbox{mybox}{colback=red!5!white,colframe=red!75!black}

\usepackage{bm} 
\usepackage{thmtools} 
\declaretheorem[style=remark]{example}

\newlength{\tabwidth}
\newlength{\tabheight}
\newlength{\tabrule}
\newlength{\tabwidthx}
\newlength{\tabheightx}

\def\gentabbox#1#2#3#4{\vbox to \tabheight{\setlength{\tabrule}{#3}%
  \setlength{\tabwidthx}{#1\tabwidth}\addtolength{\tabwidthx}{\tabrule}%

\setlength{\tabheightx}{#2\tabheight}\addtolength{\tabheightx}{-\tabheight}%
  \hbox to #1\tabwidth{%
 \hspace{-0.5\tabrule}\rule{\tabrule}{#2\tabheight}\hspace{-\tabrule}%
    \vbox to #2\tabheight{\hsize=\tabwidthx%
      \vspace{-0.5\tabrule}\hrule width\tabwidthx height\tabrule%
      \vspace{-0.5\tabrule}\vfil%
      \hbox to \tabwidthx{\hss#4\hss}%
        \vfil\vspace{-0.5\tabrule}%
      \hrule width\tabwidthx height\tabrule\vspace{-0.5\tabrule}}%
 \hspace{-\tabrule}\rule{\tabrule}{#2\tabheight}\hspace{-0.5\tabrule}}%
  \vspace{-\tabheightx}}}
\def\genblankbox#1#2{\vbox to \tabheight{\vfil\hbox to
#1\tabwidth{\hfil}}}
\def\tabbox#1#2#3{\gentabbox{#1}{#2}{0.4pt}{\strut #3}}

\catcode`\:=13 \catcode`\.=13 \catcode`\;=13
\catcode`\>=13 \catcode`\^=13
\def:#1\\{\hbox{$#1$}}
\def.#1{\tabbox{1}{1}{$#1$}}
\def>#1{\tabbox{2}{1}{$#1$}}
\def^#1{\tabbox{1}{2}{$#1$}}
\def;{\genblankbox{1}{1}\relax}
\catcode`\:=12 \catcode`\.=12 \catcode`\;=12
\catcode`\>=12 \catcode`\^=7

\newenvironment{tableau}{\bgroup\catcode`\:=13 \catcode`\.=13
  \catcode`\;=13 \catcode`\>=13 \catcode`\^=13
  \setlength{\tabheight}{3ex}\setlength{\tabwidth}{3ex}%
  \def\b##1##2##3{\gentabbox{##1}{##2}{1.2pt}{\vbox{##3}}}%
  \def\n##1##2##3{\gentabbox{##1}{##2}{0.4pt}{\vbox{##3}}}%
  \vbox\bgroup\offinterlineskip}{\egroup\egroup}

\newtheorem{theorem}{Theorem}[section]
\newtheorem{corollary}[theorem]{Corollary}
\newtheorem{lemma}[theorem]{Lemma}
\newtheorem{proposition}[theorem]{Proposition}
\newtheorem{conjecture}[theorem]{Conjecture}
\newtheorem*{theorem*}{Theorem}

\theoremstyle{definition}
\newtheorem{definition}[theorem]{Definition}

\theoremstyle{remark}

\DeclareMathOperator{\SDT}{\operatorname{SDT}}	
\DeclareMathOperator{\SSDT}{\operatorname{SSDT}}
\DeclareMathOperator{\MT}{\operatorname{MT}}
\DeclareMathOperator{\MTop}{\MT_{op}}

\DeclareMathOperator{\RS}{RS}

\DeclareMathOperator{\area}{Z}

\newcommand{\opencycles}{\mathcal{OC}}
\newcommand{\closedcycles}{\mathcal{CC}}
\newcommand{\corecycles}{\mathcal{KC}}
\newcommand{\noncorecycles}{\mathcal{NC}}

\newcommand{\oppclosedcycles}{\closedcycles^{op}}
\newcommand{\oppcorecycles}{\corecycles^{op}}
\newcommand{\oppnoncorecycles}{\noncorecycles^{op}}

\newcommand{\rxi}{\tau^\mathcal L}

\numberwithin{figure}{section}
\numberwithin{theorem}{section}

\makeatletter
\def\l@subsection{\@tocline{2}{0pt}{2.5pc}{5pc}{}}
\makeatother

\begin{document}
\renewcommand\thmcontinues[1]{Continued}
\title{Kazhdan-Lusztig left cells in type $B_n$ for intermediate parameters}

\author{Edmund Howse}
\address{Department of Mathematics \\ National University of Singapore \\ 10 Lower Kent Ridge Road \\ Singapore 119076}
\email{edmund.howse@nus.edu.sg}
\thanks{The first named author is supported by Singapore MOE Tier 2 AcRF MOE2015-T2-2-003. }

\author{Thomas Pietraho}
\address{Department of Mathematics, Bowdoin College \\Brunswick, Maine, USA 04011}
\email{tpietrah@bowdoin.edu}
\urladdr{www.pietraho.com}

\date{}

\begin{abstract}
Using a characterization of a generalized $\tau$-invariant for intermediate parameter Hecke algebras in type $B_n$ obtained in \cite{howse:vogan}, we verify a conjectural description of Kazhdan-Lusztig cells in this setting due to C.~Bonnaf\'e, L.~Iancu, M.~Geck, and T.~Lam.
\end{abstract}

\maketitle
\tableofcontents


\section{Introduction}

Hecke algebras play a central role in the representation theory of reductive algebraic groups over finite and $p$-adic fields.
Each can be constructed as a specialization of an Iwahori-Hecke algebra $\mathcal{H}$ which itself can be defined via generators and relations from a Coxeter group $W$ without explicit dependence on the underlying algebraic group.

Defined in \cite{kl:invent} and broadened by G.~Lusztig to the setting of weighted Iwahori-Hecke algebras in \cite{lusztig:leftcells} and \cite{lusztig:unequal}, Kazhdan-Lusztig cells describe partitions of $W$ known as left, right, and two-sided cells.
Their well-known classification in type $A$ motivates our work. We briefly recount the results.  The Robinson-Schensted map
$$\textup{RS}: S_n \rightarrow \textup{SYT}(n) \times \textup{SYT}(n)$$ defines a bijection between the symmetric group and same-shape pairs of standard Young tableaux.  If we write $\RS(w) = (P(w),Q(w))$, then $P(w)$ and $Q(w)$ are known as the left and right tableaux of $w$.
It follows from the work of A.~Joseph in \cite{joseph:characteristic} and \cite{joseph:towards2} on primitive ideals for complex Lie algebras of type $A$ that left cells in this setting consist of those permutations whose right tableaux agree, right cells of those permutations whose left tableaux agree, and two-sided cells of all permutations whose image tableaux have the same shape.  See \cite{ariki:rs} for a comprehensive exposition.

For weighted Coxeter groups in type $B_n$, the focus of this paper, partitions into cells depend on a single positive parameter $s$.  When $s=1$, known as the {\it equal-parameter case}, cell partitions can again be derived from the classification of primitive ideals for complex Lie algebras of type $B_n$.
Obtained by D.~Garfinkle in \cite{garfinkle1}, \cite{garfinkle2}, and \cite{garfinkle3}, the theory mimics type $A$ results, but depends on a correspondence between the Coxeter group $W_n$ of type $B_n$ and same-shape pairs of standard domino tableaux.  This time, left and right cells in $W_n$ are the pre-images of certain equivalence classes of right and left tableaux, respectively; and two-sided cells are described by equivalence classes of partitions.

Kazhdan-Lusztig cells have been determined for a few other values of the parameter $s$  in type $B_n$.  When $s=\frac{1}{2}$ and $s=\frac{3}{2}$, their classification appears already in  \cite{lusztig:leftcells} and for $s > n-1$, known as the {\it asymptotic case}, the problem was resolved by C.~Bonnaf\'e and L.~Iancu in \cite{bonnafe:iancu}. Reconciling these descriptions, C.~Bonnaf\'e, L.~Iancu, M.~Geck, and T.~Lam formulated a set of conjectures relating cells for all values of the parameter $s$ to a one-parameter family $G_r$ of domino-insertion algorithms in \cite{bgil} modeled after Garfinkle's original map defined in \cite{garfinkle1}.

In the present paper, we address the case $s = n-1$.  Using results obtained in \cite{howse:vogan}, we verify the above conjectures.
The key is an enhancement of D.~Vogan's generalized $\tau$-invariant  \cite{vogan:tau} proposed in \cite{bonnafe:geck}.  We show that as is true in type $A$ and the equal-parameter case of type $B_n$, the left-cell, generalized $\tau$-invariant, and certain tableaux-based partitions of $W$ coincide.

Our paper has the following structure.  In Section \ref{section:preliminaries}, we describe preliminaries about the Coxeter groups of type $B_n$, domino tableaux, and the domino insertion maps $G_r$. Section \ref{section:cells} recounts the definitions of Kazhdan-Lusztig cells, a conjecture of C.~Bonnaf\'e, L.~Iancu, M.~Geck, and T.~Lam, and results on the $\tau$-invariant and descent sets. The final Section \ref{section:intermediate} contains the proofs of the main results as well as motivating examples.


\section{Preliminaries}\label{section:preliminaries}


\subsection{Hyperoctahedral groups}\label{section:hyperoctahedral}

Let $W_n = W(B_n)$ be the Weyl group of type $B_n$.  We realize it as the set of signed permutations on $n$ letters, writing an element in one-line notation as $w = (w_1 \, w_2 \, \ldots w_n)$.   It is a Coxeter group $(W_n,S)$ described by the Coxeter diagram:

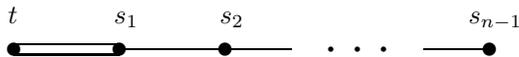
\begin{center}
\begin{picture}(300,30)
\put( 50 ,10){\circle*{5}} \put( 50, 8){\line(1,0){40}} \put( 50,
12){\line(1,0){40}} \put( 48, 20){$t$} \put( 90 ,10){\circle*{5}}
\put(130 ,10){\circle*{5}} \put(230 ,10){\circle*{5}} \put( 90,
10){\line(1,0){40}} \put(130, 10){\line(1,0){25}} \put(170,
10){\circle*{2}} \put(180, 10){\circle*{2}} \put(190,
10){\circle*{2}} \put(205, 10){\line(1,0){25}} \put( 88 ,20){$s_1$}
\put(128, 20){$s_{2}$} \put(222, 20){$s_{n-1}$}
\end{picture}
\end{center}
where we identify the generating reflections with signed permutations as $$t = (\overline{1} \, 2 \, \ldots \, n) \text{\hspace{.2in} and  \hspace{.2in}} s_i = (1 \, 2 \ldots \, i+1 \, i \ldots n).$$  We will use a bar to denote a negative entry. Write $t=t_1$ and for $2 \leq k \leq n$, let  $t_k := s_{k-1} \cdots s_1ts_1 \cdots s_{k-1}$, so that
\[t_k = ( 1 \, 2 \, \ldots \, \overline{k} \, \ldots \, n).\]


\subsection{Domino tableaux}\label{subsection:domino}

Let $\lambda$ be a partition of $n \in \mathbb{N}$ and identify it with a Young diagram $Y(\lambda)$, a left-justified array of squares whose row lengths decrease weakly.
A {\it staircase partition} takes the form $\lambda_r = [r, r-1, \ldots,  1]$ for some $r \in \mathbb{N}$.  We extend this notion, letting $\lambda_0$ denote the empty partition.

Write $s_{ij}$  for the square  lying in row $i$ and column $j$  of $Y(\lambda)$.  A {\it domino} is a pair of squares in $Y(\lambda)$ of the form $\{s_{ij},s_{i+1,j}\}$ or $\{s_{ij},s_{i,j+1}\}$; it is {\it removable} from $Y(\lambda)$ if deleting its underlying squares leaves either the empty diagram, or another Young diagram which contains the square $s_{11}$.  Starting with a Young diagram $Y(\lambda)$ we can iteratively delete removable dominos.  This process always terminates in a diagram $Y(\lambda_r)$ for some $r \geq 0$.  In fact, $r$ is determined entirely by $\lambda$ and does not depend on the precise sequence of  removable domino deletions.  We will say that $\lambda_r$ is the {\it $2$-core}, or simply the {\it core} of $\lambda$ and that $\lambda$ is a {\it partition of rank $r$}. Let $\delta_k$ denote the set of all squares $s_{ij}$ which satisfy $i+j = k + 1$.

\begin{example}  The partition $\lambda=[7,6,1^3]$ is of rank three.   Its Young diagram $Y(\lambda)$ as well as a domino tiling representing deletions of removable dominos are given below.

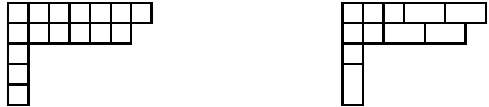
\vspace{.05in}

\begin{tiny}
$$
\begin{tableau}
:.{}.{}.{}.{}.{}.{}.{}\\
:.{}.{}.{}.{}.{}.{}\\
:.{}\\
:.{}\\
:.{}\\
\end{tableau}
\hspace{1in}
\begin{tableau}
:.{}.{}.{}>{}>{}\\
:.{}.{}>{}>{}\\
:.{}\\
:^{}\\
:;\\
\end{tableau}
$$
\end{tiny}

\end{example}

\begin{definition}   A {\it standard domino tableau of rank $r$ and shape $\lambda$} is a tiling of the non-core squares of $Y(\lambda)$ by dominos, each labeled by a unique integer from the set $\{1, \ldots , n\}$ in such a way that the labels increase along the rows and columns of $Y(\lambda)$.  We will write $\SDT_r(\lambda)$ for the set of standard domino tableaux of rank $r$ of shape $\lambda$ and $\SDT_r(n)$ for the set of standard domino tableaux of rank $r$ which contain exactly $n$ dominos.
\end{definition}

For convenience, we will label the squares in the core of $T$ with the integer 0.
The set of ordered pairs of same-shape domino tableaux in $\SDT_r(n)$ will play an important role in what follows; we will denote it as $\SSDT_r(n)$.  We will also have occasion to refer to domino tableaux that satisfy exactly the conditions above, but the set of whose squares with label 0 does not form a staircase partition. We will call these simply {\it standard domino tableaux}.


\subsection{Cycles} \label{section:cycles}

The notion of a cycle within a domino tableau was first introduced in \cite{garfinkle1} and extended to domino tableaux of rank $r$ in \cite{pietraho:rscore}.  We refer the reader to these references for details. Identifying a set of fixed squares in a tableau allows us to define
\begin{itemize}
     \item a partition of its domino labels into disjoint cycles, and
     \item the operation of moving through a cycle to produce another domino tableau.
\end{itemize}
There are two natural choices for the set of fixed squares for a rank $r$ tableau, each producing a distinct cycle partition.
\begin{definition}  Consider a tableau $T \in \SDT_r(n)$.
    \begin{enumerate}
        \item[(a)] If fixed squares in $T$ are defined as those $s_{ij}$ for which $i+j$ has the opposite parity from $r$, we will call the resulting cycles of $T$ {\it regular cycles}, or simply {\it cycles}.
        \item[(b)] On the other hand, if fixed squares in $T$  are those $s_{ij}$ for which $i+j$ has the same parity as $r$, we will call the resulting cycles of $T$ {\it opposite cycles.}
    \end{enumerate}
\end{definition}

In prior work, for instance \cite{pietraho:rscore} and \cite{pietraho:equivalence}, it has only been necessary to address regular cycles for a rank $r$ tableau, but opposite cycles will play an important role herein.

\begin{example}[label=ex:cycles]
Consider the following standard domino tableaux of rank 2:
\vspace{.08in}

\begin{tiny}
$$
\raisebox{.1in}{{\normalsize $S =$}}
\hspace{.1in}
\begin{tableau}
:.{0}.{0}>{1}\\
:.{0}^3^4\\
:^2\\
\end{tableau}
\hspace{.3in}
\raisebox{.1in}{{\normalsize $T =$}}
\hspace{.1in}
\begin{tableau}
:.{0}.{0}>{1}\\
:.{0}>2\\
:^3>4\\
\end{tableau}
$$
\end{tiny}

\noindent
Each domino label forms its own regular cycle in both $S$ as well as in $T$.  The family of opposite cycles in $S$ is $\{\{1\},\{2\},\{3,4\}\},$ while in $T$ it is  $\{\{1\},\{2,4\},\{3\}\}$.
\end{example}

Given a cycle $c$ in $T$, the moving through map constructs another domino tableau $\MT(T,c)$ that differs from $T$ only in the labels of the variable squares of $c$. Disjoint cycles can be moved though independently, and if $U$ is a set of cycles, we will write $\MT(T,U)$ for the domino tableau obtained by simultaneously moving through all of them.  When discussing tableaux pairs,  define  $$\MT((S,T),(U,V)):=(\MT(S,U),\MT(T,V)).$$
When $c$ is an opposite cycle, we will write $\MTop(T,c)$ to emphasize the opposite choice of fixed squares is being used to define the operation.  We refer the reader to \cite[\S 2.2]{pietraho:rscore} for the detailed definitions of the moving through map.

Cycles in a tableau come in a few distinct flavors depending on whether or not the moving through operation changes the underlying Young diagram and which squares are affected.
A cycle for which moving through changes the shape of the underlying tableau is called an {\it open cycle}, otherwise it is {\it closed}.  An open cycle is a {\it core cycle} if moving through it changes the total number of squares in the underlying Young diagram; it is {\it non-core} otherwise.  Respectively, we will write
\begin{center}
$\closedcycles(T),$ $\opencycles(T),$ $\corecycles(T),$ $\text{ and } \noncorecycles(T)$
\end{center}
for the sets of closed, open, core, and non-core cycles in a domino tableau $T$.  Note that $\opencycles(T)$ is a disjoint union of  $\corecycles(T)$ and  $\noncorecycles(T)$.  Our notation for the corresponding sets of opposite cycles will include the superscript $op$.

\begin{example}[continues=ex:cycles]
First we consider the cycle partition for  both $S$ and $T$.  Each of $S$ and $T$ has three core cycles and one non-core cycle:
\vspace{.07in}
\begin{center}
$\corecycles(S) = \corecycles(T) = \{\{1\},\{2\},\{3\}\}$, while
 $ \noncorecycles(S)=\noncorecycles(T) =\{4\}$.
\end{center}
\vspace{.07in}
\noindent
In the opposite cycle partition, each tableau has two core cycles and one closed cycle:
\setlength{\abovedisplayskip}{-10pt}
\setlength{\belowdisplayskip}{-10pt}
\begin{center}
\begin{alignat*}{2}
\oppcorecycles(S)  & = \{\{1\},\{2\}\}, \hspace{.07in} & \oppclosedcycles(S)  = \{\{3,4\}\},\\
\oppcorecycles(T) & =  \{\{1\},\{3\}\}, \hspace{.07in} & \oppclosedcycles(T)  = \{\{2,4\}\}.\\
\end{alignat*}
\end{center}
Moving through all the  core cycles of a standard domino tableau of rank $r$ results in a standard domino tableau of rank $r+1$.  Similarly, moving through all of the opposite core cycles reduces the rank of a standard domino tableau by one.  For instance, with $S$ and $T$ as above,

\vspace{.1in}
\begin{tiny}
$$
\raisebox{.15in}{{\normalsize $\MT(S,\corecycles(S)) =$}}
\hspace{.1in}
\begin{tableau}
:.{0}.{0}.{0}>{1}\\
:.{0}.{0}^4\\
:.{0}^3\\
:^{2}\\
\end{tableau}
\hspace{.3in}
\raisebox{.15in}{{\normalsize \hspace{.1in} and \hspace{.1in} $\MT(T,\corecycles(T)) =$}}
\hspace{.1in}
\begin{tableau}
:.{0}.{0}.{0}>{1}\\
:.{0}.{0}>2\\
:.{0}>4\\
:^3\\
\end{tableau}.
$$
\end{tiny}

\vspace{.1in}
\noindent
In general, for $(S,T) \in \SSDT_r(n)$ the shape of $\MT(S,\corecycles(S))$ will be different from the shape of $\MT(T,\corecycles(T))$, as occurs above.  However, it is possible to amend this by including additional non-core open cycles in the moving though map.  Again looking at regular cycles in the tableaux above, let $\gamma(S,T) = \corecycles(S) \cup \{4\}$ and $\gamma(T,S) = \corecycles(T) \cup \{4\}$, then moving through produces a shape-shape pair of standard domino tableaux:

\vspace{.1in}

\begin{tiny}
$$
\raisebox{.15in}{{\normalsize $\MT(S,\gamma(S,T)) =$}}
\hspace{.1in}
\begin{tableau}
:.{0}.{0}.{0}>{1}\\
:.{0}.{0}>4\\
:.{0}^3\\
:^{2}\\
\end{tableau}
\hspace{.3in}
\raisebox{.15in}{{\normalsize \hspace{.1in} and \hspace{.1in} $\MT(T,\gamma(T,S)) =$}}
\hspace{.1in}
\begin{tableau}
:.{0}.{0}.{0}>{1}\\
:.{0}.{0}>2\\
:.{0}^4\\
:^3\\
\end{tableau}.
$$
\end{tiny}
\end{example}

\vspace{.1in}

\begin{definition}
Consider a pair $(S,T) \in \SSDT_r(n)$. Let $\gamma(S,T)$ and $\gamma(T,S)$ be the minimal sets of open cycles in $S$ and $T$, respectively, that satisfy:
\begin{enumerate}
     \item[(a)] $\corecycles(S) \subset \gamma(S,T)$, $\corecycles(T) \subset \gamma(T,S)$, and
     \item[(b)] the tableaux $\MT(S,\gamma(S,T))$ and $\MT(T,\gamma(T,S))$ have the same shape.
\end{enumerate}
As in  \cite[2.3.1]{garfinkle2}, we refer to $\gamma(S,T)$ as the set of {\it extended open cycles in $S$ relative to $T$ through $\corecycles(S)$} and define $\gamma(T,S)$ similarly.  Finally, write $\gamma_{ST}$ for the pair $(\gamma(S,T), \gamma(T,S))$.
\end{definition}

As detailed in \cite{pietraho:rscore}, the above discussion allows us to define a bijective map on same-shape domino tableau pairs
$\Gamma_r : \SSDT_r(n) \rightarrow \SSDT_{r+1}(n)$ by letting
$$\Gamma_r((S,T)) = \MT((S,T), \gamma_{ST}).$$


\subsection{Domino insertion} \label{section:dominoinsertion}

Modeled on the Robinson-Schensted algorithm, \cite{garfinkle1} and \cite{vanleeuwen:rank} introduced a one-parameter family of bijections
$$G_r: W_n \rightarrow \SSDT_r(n).$$
In its original form, these maps are defined by using a domino insertion algorithm where a pair of standard domino tableaux is constructed simultaneously, the left tableaux by inserting and bumping dominos starting with a Young diagram of the staircase partition $\lambda_r$, while the right tracking the shape of the insertion tableau.  We will write $G_r(w) = (P_r(w),Q_r(w))$ for the image of an element of $W_n$ under this map and extend this notation writing $G_r^k(w) = (P_r^k(w), Q_r^k(w))$ for the tableaux obtained after $k$ insertion steps.

\begin{example}  Consider the signed permutation $w=(4 \, 1 \, \overline{3} \, \overline{2}) \in W_4$.  The following sequence of tableau pairs represents its images under $G_r$ for increasing values of $r$.
\vspace{.05in}
\begin{tiny}
$$
\hspace{-.2in}
\raisebox{.1in}{{\normalsize $P_0(w)=$}}
\hspace{.1in}
\begin{tableau}
:>{1}^4\\
:^2^3\\
:;\\
\end{tableau}
\hspace{.7in}
\raisebox{.1in}{{\normalsize $Q_0(w) =$}}
\hspace{.1in}
\begin{tableau}
:>{1}^4\\
:>2\\
:>3\\
\end{tableau}
$$
\end{tiny}

\begin{tiny}
$$
\hspace{-.2in}
\raisebox{.1in}{{\normalsize $P_1(w)=$}}
\hspace{.1in}
\begin{tableau}
:.{0}>{1}\\
:^2^3^4\\
:;\\
\end{tableau}
\hspace{.7in}
\raisebox{.1in}{{\normalsize $Q_1(w) =$}}
\hspace{.1in}
\begin{tableau}
:.{0}>{1}\\
:>2^4\\
:>3\\
\end{tableau}
$$
\end{tiny}

\begin{tiny}
$$
\hspace{-.1in}
\raisebox{.1in}{{\normalsize $P_2(w)=$}}
\hspace{.1in}
\begin{tableau}
:.{0}.{0}>{1}\\
:.{0}^3^4\\
:^2\\
\end{tableau}
\hspace{.6in}
\raisebox{.1in}{{\normalsize $Q_2(w) =$}}
\hspace{.1in}
\begin{tableau}
:.{0}.{0}>{1}\\
:.{0}>2\\
:^3>4\\
\end{tableau}
$$
\end{tiny}

\begin{tiny}
$$
\raisebox{.1in}{{\normalsize $P_3(w)=$}}
\hspace{.1in}
\begin{tableau}
:.{0}.{0}.{0}>{1}\\
:.{0}.{0}>4\\
:.{0}^3\\
:^2\\
\end{tableau}
\hspace{.5in}
\raisebox{.1in}{{\normalsize $Q_3(w) =$}}
\hspace{.1in}
\begin{tableau}
:.{0}.{0}.{0}>{1}\\
:.{0}.{0}>2\\
:.{0}^4\\
:^2\\
\end{tableau}
$$
\end{tiny}

When $r \geq n-1$, $G_r$ recovers a well-known bijection  between signed permutations and same-shape pairs of standard Young bitableaux, see \cite{bonnafe:iancu}.  Positive values are inserted as horizontal dominos at the top of the tableau and negative values as vertical dominos at the bottom.  If $r \geq n-1$, the two sets of labels do not interact during insertion as is true in the bitableaux algorithm.  Tableaux with this latter property will play a role in what follows, so we define:

\begin{definition}\label{definition:split}  A standard domino tableau $T$ of rank $r$ is {\it split} iff its diagonal $\delta_{r+2}$ contains at least one empty square. We extend this notion in the natural way to same-shape tableau pairs.
\end{definition}

There is a simple relationship among the tableau pairs in $G_r(w)$ for the different values of $r$.  It can be described completely in terms of extended cycles. With $\Gamma_r$ defined as in Section \ref{section:cycles}, we have:

\end{example}

\begin{theorem}[\cite{pietraho:rscore}]\label{proposition:MMT}
\hspace{.1in} $\Gamma_r(G_r(w)) = G_{r+1}(w).$
\end{theorem}

The inverse of $\Gamma_r$, which we write as $\Upsilon_r$, is defined on a tableau pair $(S,T)$ by moving through the extended cycles $\upsilon_{ST} = (\upsilon(S,T), \upsilon(T,S))$, where $\upsilon(S,T)$ is the extended opposite open cycle in $S$ relative to $T$ through $\oppcorecycles(S)$ and $\upsilon(T,S)$ is the extended opposite open cycle in $T$ relative to $S$ through $\oppcorecycles(T)$.

We will be particularly interested in equivalence classes on $W_n$ defined by domino tableaux.  The simplest are those defined by having the same right tableaux:

\begin{definition}  For $T \in \SDT_r(n)$, let
$$\mathcal{C}(T) = \{ w \in W_n \; | \; Q_r(w) = T\}.$$
\end{definition}

\noindent
For each value of $r$, these classes define a partition of $W_n$.  As $r$ varies, there is an intricate, but tractable, relationship among them.  We describe it in the following proposition.

\begin{proposition}\label{proposition:MTDelta}  Let $T \in \SDT_r(n)$ and define $T' := \MT(T,\corecycles(T)) \in \SDT_{r+1}(n)$.  If $U$ is a set of regular cycles in $T$, then it is also a set of opposite cycles in $T'$.  Write $T_U := \MT(T,U)$ and $T'_U := \MTop(T',U)$.  Then

$$\bigsqcup_{U \subset \noncorecycles(T)} \mathcal{C}(T_U)  \hspace{.1in} = \bigsqcup_{U \subset \noncorecycles(T)}\mathcal{C}(T'_U).$$
\end{proposition}

\begin{proof}  First note the following two sets of simple relationships between sets of open cycles among the above tableaux:
$\corecycles(T_U) = \corecycles(T) =\oppcorecycles(T')$ and   $\noncorecycles(T_U) = \noncorecycles(T) =\oppnoncorecycles(T').$
We will work with sets of same-shape pairs of standard domino tableaux instead of subsets of $W_n$.
Write $\mathcal{D}(T):=G_r(\mathcal{C}(T))$  for the set of same-shape rank $r$ standard domino tableaux pairs whose right tableaux is $T$.

Let $$X(T) := \bigsqcup_{U \subset \noncorecycles(T)} \mathcal{D}(T_U) \text{\hspace{.1in} and \hspace{.1in}} Y(T) := \bigsqcup_{U \subset \noncorecycles(T)}\mathcal{D}(T'_U).$$
We will show that $\Gamma_r(X(T)) = Y(T)$.
Hence consider $(S,T_U) \in \SSDT_r(n)$. By definition of $\Gamma_r$, the right tableau of $\Gamma_r(S,T_U)$ equals $\MT(T_U,\gamma(T_U,S))$.
The extended cycle $\gamma(T_U,S)$ in $T_U$ consists of all the open cycles in $\corecycles(T_U) = \corecycles(T)$ together with a subset of non-core open cycles $U' \subset \noncorecycles(T_U) = \noncorecycles(T)$ which depends on $S$.  If we write $\vartriangle$ for the symmetric-difference relation on sets,
our observation means that $\Gamma_r(S,T_U) \in \mathcal{D}(T'_{U \vartriangle U'})$. Thus we have
$\Gamma_r(X(T)) \subset Y(T)$.

Now consider $(S', T'_U) \in \SSDT_{r+1}(n)$.  Then the right tableau of $\Upsilon_r(S',T'_U)$
equals $\MT(T'_U, \upsilon(T_U',S'))$ where $\upsilon(T_U',S')$ consists of all the opposite core open cycles in  $\oppcorecycles(T_U') = \corecycles(T)$
together with a subset of opposite non-core open cycles $U' \subset \oppnoncorecycles(T') = \noncorecycles(T)$ which depends on $S'$.  This time, $\Upsilon_r(S',T'_U) \in \mathcal{D}(T_{U \vartriangle U'})$ and we have $X(T) \supset \Upsilon_r(Y(T))$.
Together, the two inclusions give us the desired equality.
\end{proof}

In Section \ref{section:intermediate}, we will use the above proposition in two especially tractable cases, when the set $\noncorecycles(T)$ of non-core open cycles is either empty, or contains exactly one cycle.


\section{Cells}\label{section:cells}
Following G.~Lusztig's construction in \cite{lusztig:unequal} and the combinatorics of standard domino tableaux detailed in \cite{bgil} and \cite{pietraho:equivalence}, we describe two one-parameter families of partitions of $W_n$, as well as a conjecture relating them.


\subsection{Kazhdan-Lusztig cells}
The first partition is defined in terms of the algebraic structure of a two-parameter algebra derived from $W_n$.
Let $S$ be the set of simple roots of $W_n$ as defined in Section \ref{section:hyperoctahedral}.
Write $\ell: W_n \rightarrow \mathbb{Z}_{\geq 0}$ for the standard length function computed using the generators in $S$. A weight function $\mathcal{L}: W_n \rightarrow \mathbb{Z}$ is a map satisfying $\mathcal{L}(wy) = \mathcal{L}(w) + \mathcal{L}(y)$ whenever $\ell(wy) = \ell(w) + \ell(y)$.  It is uniquely determined by its values on the simple roots $S$ and is subject to the condition that whenever $s, s' \in S$ and $s s'$ is of odd order, $\mathcal{L}(s) = \mathcal{L}(s')$.  We will let $\mathcal{L}(t) = b$ and $\mathcal{L}(s_i) = a$ for all $i$ with $a,b \in \mathbb{N}$ and write $\mathcal{L} = \mathcal{L}(a,b)$.

Let $\mathcal{H}_n$ be the generic
Iwahori-Hecke algebra over $\mathcal{A}= \mathbb{Z}[v, v^{-1}]$ with
parameters $\{v_s \, | \, s \in S\}$, where  $v_w = v^{\mathcal{L}(w)}$ for
all $w \in W_n$. The algebra $\mathcal{H}_n$ is free over $\mathcal{A}$
and has a basis $\{T_w \, | \, w \in W\}$ in terms of which
multiplication takes the form
$$T_s T_w = \left\{
        \begin{array}{ll}
            T_{sw} & \text{if $\ell(sw) > \ell(w)$, and}\\
            T_{sw}+(v_s-v_s^{-1}) T_w & \text{if $\ell(sw) < \ell(w)$,}
        \end{array}
        \right.
        $$
for $s \in S$ and $w \in W_n$.  As detailed in \cite{lusztig:unequal}, each choice of weight function $\mathcal{L}$, or equivalently, each pair $a,b \in \mathbb{N}$, defines  partitions of $W_n$ into left, right, and two-sided cells.  In fact, these partitions depend only on the ratio $\frac{b}{a}$ of the parameters. In what follows, we will only consider integer values of this ratio.

\newtcolorbox{mybox2}{colback=black!10!white,colframe=white!90!black}
\begin{mybox2}
{\bf Notation:}
We will restrict our attention to $\tfrac{b}{a} \in \mathbb{N}$,
write $$r = \tfrac{b}{a}-1,$$ and call  components of the corresponding partitions {\it left, right, and two-sided $r$-cells}.  Denote the resulting equivalence relations by $\approx_r^L$, $\approx_r^R$, and $\approx_r^{LR}$, respectively, and let $\mathcal{K}_r^L(w)$, $\mathcal{K}_r^R(w)$, and $\mathcal{K}_r^{LR}(w)$
be the cells containing the element $w \in W_n$.  We will focus almost exclusively on left cells and will omit the superscript in our notation unless there is a potential for confusion.
\end{mybox2}


\subsection{Combinatorial cells}

Inspired by the classification of Kazhdan-Lusztig cells for classical Weyl groups in terms of standard Young and domino tableaux, we have the following definition of a family of partitions of $W_n$ based on the images of the maps $G_r$.

\begin{definition}  Consider an non-negative integer $r$ and let $S, T \in \SDT_{r}(n)$.  We will write $S \sim_r T$ if and only if there is a set of non-core open cycles $U \subset \noncorecycles(T)$ such that $S = \MT(T,U)$.  If $w,y \in W_n$, we will say
\begin{enumerate}
    \item[(a)] $w \sim^L_r y$ if and only if  $Q_r(w) \sim_r Q_r(y),$
    \item[(b)] $w \sim^R_r y$ if and only if  $P_r(w) \sim_r P_r(y),$
    \item[(c)] $w \sim^{LR}_r y$ if and only if there is a sequence $w = w_0$, $w_1$, $\ldots$, $w_k = y$ where for all $i < k$, either $w_i \sim^L_r w_{i+1}$ or $w_i \sim^R_r w_{i+1}$.
\end{enumerate}
\end{definition}

\begin{mybox2}
{\bf Notation:} We will call the equivalence classes on $W_n$ defined by the relations $\sim^L_r$, $\sim^R_r$, and $\sim^{LR}_r$ {\it left, right,} and {\it two-sided combinatorial $r$-cells} writing $\mathcal{C}^L_r(w)$, $\mathcal{C}^R_r(w)$, and $\mathcal{C}^{LR}_r(w)$ for  the cells containing the element $w \in W_n$.  Again, we will omit the superscript for left cells in our notation unless there is a potential for confusion.
\end{mybox2}

When $r \geq n-1$ it is easy to see from the definition of the insertion algorithm $G_r$ that all combinatorial left $r$-cells and $(r+1)$-cells coincide.  The same is true for combinatorial right and two-sided cells.  We will call these cells {\it asymptotic} and write $\mathcal{C}^L_a(w)$, $\mathcal{C}^R_a(w)$, and $\mathcal{C}^{LR}_a(w)$, usually omitting the superscript for left cells.

Left combinatorial $r$-cells can be written as unions of  sets of the form $\mathcal{C}(T)$ for rank $r$ tableaux.
Using notation from Section \ref{section:dominoinsertion}, we have $$\mathcal{C}_r(w) = \bigsqcup_{U \in \noncorecycles(T)} \mathcal{C}(T_U)$$ where $T = Q_r(w)$ and $T_U := \MT(T,U)$.
At the heart of this paper is the conjecture of C.~Bonnaf\'e, L.~Iancu, M.~Geck, and T.~Lam which, when interpreted in light of the results from \cite{pietraho:equivalence}, states that combinatorial $r$-cells agree with  Kazhdan-Lusztig $r$-cells on $W_n$:

\begin{conjecture}[\cite{bgil}]\label{conjecture:bgil}
For $w \in W_n$ and $r \in \mathbb{Z}_{\geq 0}$, $\mathcal{C}_r(w) = \mathcal{K}_r(w).$
\end{conjecture}

Among integral values of $r$, this conjecture has been verified for $r=0$ in \cite{garfinkle3} and for $r \geq n-1$ in \cite{bonnafe:iancu}, the latter in particular implying that $\mathcal{C}_a(w) = \mathcal{K}_a(w)$ for all $w \in W_n$.  Our work concerns $r = n-2$, which we follow \cite{howse:vogan} in calling the {\it intermediate parameter case} and will refer to $r$-cells as {\it intermediate cells}.


\subsection{Descent set and an enhanced $\tau$-invariant}

For a Coxeter group $W$ with generating reflections $S$, the {\it right descent set}, or the {\it right $\tau$-invariant}, of an element $w \in W$ is the set
$$\tau(w) = \{ s \in S \; | \; \ell(ws) < \ell(w) \}.$$
We avoid handedness and refer to it simply as the $\tau$-invariant of $w$.  In the type $B_n$ setting where the weight function has unequal parameters, \cite{pietraho:knuth} and \cite{howse:vogan} extended the $\tau$-invariant to draw from not only simple reflections but also some of the elements of the form $t_j$ defined in Section \ref{section:hyperoctahedral}.  More precisely,

\begin{definition}[\cite{pietraho:knuth} and \cite{howse:vogan}] For $w \in W_n$ and weight function $\mathcal{L} =\mathcal{L}(a,b)$, the enhanced $\tau$-invariant is the set
\[\rxi(w) = \tau(w) \cup\{ t_j  \;|\;  \tfrac{b}a > j-1  \ \textup{ and }\  \ell(wt_j) < \ell(w) \}.\]
\end{definition}
\noindent
When the parameters $a$ and $b$ of the weight function $\mathcal{L}(a,b)$ are equal, then this definition recovers the usual $\tau$-invariant and
$\rxi(w) = \tau(w).$

It is possible read-off the enhanced $\tau$-invariant  both from the signed-permutation representation of $w$ as well as from the domino tableaux $G_r(w)$, although in the latter case this is only straightforward for weight functions with certain parameters.  The following instructs us how to proceed in the case of signed permutations:

\begin{proposition}  Let $w \in W_n$.  Then
\begin{enumerate}
    \item[(a)]  $\ell(ws_i) < \ell(w) \iff w(i+1) < w(i)$, and
    \item[(b)]  $\ell(wt_j) < \ell(w) \iff w(j)<0$.
\end{enumerate}
\end{proposition}

To compute $\rxi(w)$ from the tableaux $G_r(w)$, we first imitate \cite[Section 2.1]{garfinkle2} and define the notion of an enhanced $\tau$-invariant for a domino tableau $Q$ of rank $r$.  Inclusion in this set will be determined by the relative positions of the dominos in $Q$.  We restrict the definition to the setting where the weight function $\mathcal{L}=\mathcal{L}(a,b)$ satisfies $r \geq \frac{b}{a}-1$.  Note that in the equal-parameter case, this inequality is satisfied by all values of $r$.



\begin{definition}
Consider $Q \in \SDT_r(n)$ and let $\mathcal{L}=\mathcal{L}(a,b)$ with $r \geq \frac{b}{a}-1$. Then

\begin{itemize}
\item[(a)] $t_j \in \rxi(Q)$ if and only if $D(j,Q)$ is vertical for $1 \leq j \leq r+1$, and
\item[(b)] $s_i \in \rxi(Q)$ if and only if $D(i,Q)$ lies strictly above $D(i+1,Q)$ for $1 \leq i \leq n$.
\end{itemize}

\noindent Extending this definition to pairs of domino tableaux, we let  $\rxi((P,Q)) := \rxi(Q)$.  The ordinary $\tau$-invariant $\tau(Q)$ and $\tau(P,Q)$ is defined by taking the intersection of this set with $S$.
\end{definition}

The following proposition relating the enhanced $\tau$-invariant for signed permutations and domino tableaux generalizes \cite[2.1.18]{garfinkle2} where it is stated for the regular $\tau$-invariant when $r=0$.   It is possible to verify it by carefully extending the original proof to domino tableaux of rank $r$.  Instead we will take a more streamlined approach using the maps $\Gamma_r$ which relate the domino insertion maps for tableaux of different ranks.

\begin{proposition}\label{proposition:enhanced tau}
Let $w \in W_n$ and suppose that $\mathcal{L}=\mathcal{L}(a,b)$ where $\frac{b}{a}$ is an integer and $r \geq \frac{b}{a}-1$.  Then $$\rxi(w) = \rxi(G_r(w)).$$
\end{proposition}

This proposition has the following immediate corollary for the regular $\tau$-invariant:
\begin{corollary} \label{proposition:tau} Let $w \in W_n$.  For all $r\geq 0$, we have
$$\tau(w) = \tau(G_r(w)).$$
\end{corollary}

Our proof relies on two intermediate results.  The first will help us understand the effect of the maps $\Gamma_r$ on the ordinary $\tau$-invariant.

\begin{lemma}\label{lemma:tau:cycle}  Suppose that $T$ is a standard domino tableau and $c$ is either an open cycle in $T$ or a closed cycle not of the form $c=\{i,i+1\}$.  Then $\tau(T) = \tau(\MT(T,c))$.  The same holds for corresponding opposite cycles in $T$ as long as $T$ is not a standard domino tableau of rank $0$.
\end{lemma}
\begin{proof}
This is a generalization of \cite[3.1.4]{garfinkle3}, allowing both general rank tableaux as well as certain core open cycles, which the original result omits.  Thus consider a cycle or opposite cycle $c$ in $T$ which satisfies the conditions of the proposition.  Let $T' = \MT(T,c)$ or $T' = \MTop(T,c)$ . We will show that $\tau(T) \subset \tau(T')$.  Since moving through is an involution, this will suffice.

First suppose that $t \in \tau(T)$.  Then $D(1,T)$ is vertical, and unless $s_{11} \in D(1,T)$ and $c$ is an opposite cycle, so will be $D(1,T')$.  Thus $t \in \tau(T')$.  If $s_i \in \tau(T)$, then we have to examine the relative positions of $D(i,T)$ and $D(i+1,T)$.
If we assume that in fact $s_i \notin \tau(T')$, then $D(i,T')$ and $D(i+1,T')$ must be adjacent and their relative positions must be one of the following:

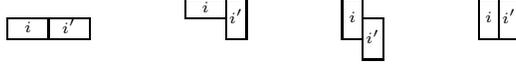
\begin{tiny}
$$
\hspace{-.2in}
\hspace{.1in}
\begin{tableau}
:>{i}>{i'}\\
:;\\
\end{tableau}
\hspace{.4in}
\hspace{.1in}
\begin{tableau}
:>{i}^{i'}\\
:;\\
:;\\
\end{tableau}
\hspace{.4in}
\hspace{.1in}
\begin{tableau}
:^{i}\\
:;^{i'}\\
:;\\
\end{tableau}
\hspace{.4in}
\hspace{.1in}
\begin{tableau}
:^{i}^{i'}\\
:;\\
:;\\
\end{tableau}
$$
\end{tiny}
\hspace{-.12in}
where we have written $i':=i+1$.  We eliminate each possibility successively from left to right, writing $s_{kl}$ for the top, leftmost square of $D(i,T')$. The first possibility implies that the fixed squares with labels $i$ and $i+1$ lie in the same row in both $T$ and $T'$, which contradicts our original assumption that $s_i \in \tau(T)$.  The second and third imply impossible labellings of a standard domino tableaux: $s_{k+1,l+1}$ in the second and $s_{k,l+1}$ in the third would have to be labeled with a integer between $i$ and $i+1$.  For the final diagram to have arisen from a tableau $T$ with $s_i \in \tau(T)$, $s_{kl}$ must be a fixed square.  But in this case, $c = \{i,i+1\}$ is a closed cycle, which we have excluded from consideration.
\end{proof}

The following statement resolves the inclusion of the $t_j$ in the enhanced $\tau$-invariant for domino tableaux of sufficiently large rank.  It follows directly from the definition of the domino insertion maps $G_r$.

\begin{lemma}\label{lemma:tfae}
Let $w \in W_n$ and suppose that $\mathcal{L}=\mathcal{L}(a,b)$ where $\frac{b}{a}$ is an integer and $r \geq \frac{b}{a}-1$.  When $k \leq r+1$, the following statements are equivalent for all $1 \leq j \leq k$:
\begin{itemize}
\item[(a)] $D(w(j),P^k(w))$ is horizontal,
\item[(b)] $D(j,Q^k(w))$ is horizontal,
\item[(c)] $w(j) > 0$, and
\item[(d)] $t_j \notin \rxi(w)$.
\end{itemize}
\end{lemma}

%
%

\begin{proof}[Proof of Proposition~\ref{proposition:enhanced tau}]

The original version of this result \cite[2.1.9]{garfinkle2} gives us the desired statement for the ordinary $\tau$-invariant and the map $G_0$.  Since for a domino tableau pair $(P,Q)$ of rank $r$, the rank $r+1$ tableaux $\Gamma_r(P,Q)$ are defined by moving through a set of open cycles in $P$ and $Q$, and by Proposition \ref{proposition:MMT}, $G_{r+1} = \Gamma_r \circ G_r$ for all $r$,  Lemma \ref{lemma:tau:cycle} implies that our result holds for the ordinary $\tau$-invariant for all maps $G_r$.

To verify the proposition for the enhanced $\tau$-invariant, we need only compare membership of $t_j$ for $2 \leq j \leq r+1$ in $\rxi(w)$ and $\rxi(G_r(w))$.  Note that the restriction of the recording tableau $Q_r(w)$ to its core of rank $r$ and the dominos with labels in the set $\{1,2,\ldots,r+1\}$ is a split domino tableau. We may now apply Lemma~\ref{lemma:tfae} to conclude the proof.
\end{proof}


\section{Intermediate cells}\label{section:intermediate}

We are ready to address the conjecture of Bonnaf\'e, Geck, Iancu, and Lam.
Using the inverse of the rank-increasing map $\Gamma_r$ on tableaux pairs, we first describe the relationship between  Kazhdan-Lusztig asymptotic left cells and  combinatorial left intermediate cells, that is, cells for the parameter $r = n-2$.  We show that each  combinatorial left intermediate cell is either an asymptotic left cell itself or a union of two such cells.  When reconciled with the results on Kazhdan-Lusztig left intermediate cells in \cite{howse:vogan}, this allows us to deduce Conjecture \ref{conjecture:bgil} in this setting.

Throughout most of this section we will work in  a more general environment obtaining results for certain combinatorial left cells for all values of $r$. We begin with two examples that illustrate the issues involved.


\subsection{Examples}
The first example concerns a set of tableaux for which rank-$r$ and rank-$(r+1)$ left combinatorial cells coincide.  In the second example, rank-$r$ left combinatorial cells are unions of rank-$(r+1)$ left combinatorial cells.

\begin{example}\label{example:nonarea}
Consider the following tableaux pair $(S,T) \in \SSDT_2(5)$:

\begin{tiny}
$$
\raisebox{.10in}{{\normalsize $S =$}}
\hspace{.1in}
\begin{tableau}
:.{}.{}>{1}>4\\
:.{}>{3}>5\\
:^2\\
\end{tableau}
\hspace{.3in}
\raisebox{.1in}{{\normalsize $T =$}}
\hspace{.1in}
\begin{tableau}
:.{}.{}>{3}>4\\
:.{}>{2}>5\\
:^1\\
\end{tableau}.
$$
\end{tiny}
\vspace{.05in}

\noindent
For both tableaux, each row of horizontal dominos and each column of vertical dominos forms a core  open cycle. Each also forms its own extended cycle.  In fact, recalling notation from Section \ref{section:dominoinsertion} for extended cycles,
$\gamma(S,T) = \{\{1,4\},\{2\},\{3,5\}\}$ and $ \gamma(T,S) = \{\{1\},\{2,5\},\{3,4\}\}.$
Consequently, $\Gamma_2(S,T) = (S',T')$ where:

\vspace{.1in}
\begin{tiny}
$$
\raisebox{.17in}{{\normalsize $S' =$}}
\hspace{.1in}
\begin{tableau}
:.{}.{}.{}>{1}>4\\
:.{}.{}>{3}>5\\
:.{}\\
:^2\\
\end{tableau}
\hspace{.3in}
\raisebox{.17in}{{ \normalsize and \hspace{.2in}  $T' =$}}
\hspace{.1in}
\begin{tableau}
:.{}.{}.{}>{3}>4\\
:.{}.{}>{2}>5\\
:.{}\\
:^1\\
\end{tableau}.
$$
\end{tiny}

\noindent
It is easy to see that the extended cycles $\gamma(T,S)$ in $T$ do not depend on $S$.
The tableaux $T$ and $T'$  contain no non-core open cycles, thus $\mathcal{C}(T)$ is a combinatorial left 2-cell and $\mathcal{C}(T')$ is a combinatorial left  3-cell.  Proposition \ref{proposition:MTDelta} in fact implies that  we have equality $\mathcal{C}(T) = \mathcal{C}(T')$.
\end{example}

\begin{example} \label{example:area}
Next, consider the tableaux pair $(S,T) \in \SSDT_2(4)$:

\vspace{.1in}

\begin{tiny}
$$
\raisebox{.1in}{{\normalsize $S =$}}
\hspace{.1in}
\begin{tableau}
:.{}.{}>{1}\\
:.{}^3^4\\
:^2\\
\end{tableau}
\hspace{.3in}
\raisebox{.1in}{{\normalsize  $T =$}}
\hspace{.1in}
\begin{tableau}
:.{}.{}>{1}\\
:.{}>2\\
:^3>4\\
\end{tableau}.
$$
\end{tiny}

\noindent
In both tableaux, each domino forms its own open cycle.  Extended cycles are more intricate with
$\gamma(S,T) = \{\{1\},\{2\},\{3,4\}\}$ and $\gamma(T,S) = \{\{1\},\{2,4\},\{3\}\}.$
Consequently  $\Gamma_2(S,T) = (S',T')$ where

\vspace{.1in}

\begin{tiny}
$$
\raisebox{.17in}{{\normalsize $S' =$}}
\hspace{.1in}
\begin{tableau}
:.{}.{}.{}>{1}\\
:.{}.{}>4\\
:.{}^3\\
:^2\\
\end{tableau}
\hspace{.3in}
\raisebox{.17in}{{\normalsize and \hspace{.2in} $T' =$}}
\hspace{.1in}
\begin{tableau}
:.{}.{}.{}>{1}\\
:.{}.{}>2\\
:.{}^4\\
:^3\\
\end{tableau}.
$$
\end{tiny}

\noindent
This time, the extended cycles in $\gamma(T,S)$ in $T$ do depend on the tableau $S$.
Note that since $\{4\}$ is an open cycle in $T$, then it is also an opposite open cycle in $T'$ and write $\overline{T} = \MT(T, \{4\})$ as well as $\overline{T'} = \MTop(T', \{4\})$.  Since $\{4\}$ is the only non-core open cycle in $T$, we know that $\mathcal{C}(T) \cup \mathcal{C}(\overline{T})$ is a  combinatorial left 2-cell.  Neither $T'$ nor $\overline{T'}$ have non-core open cycles, thus
$\mathcal{C}(T')$ and  $\mathcal{C}(\overline{T'})$ each form a combinatorial left 3-cell.  By Proposition \ref{proposition:MTDelta},
$$\mathcal{C}(T) \cup \mathcal{C}(\overline{T}) = \mathcal{C}(T') \cup \mathcal{C}(\overline{T'})$$
which we read as a decomposition of a combinatorial left 2-cell into two  combinatorial left 3-cells.
\end{example}


\subsection{A partition of $W_n$}\label{subsection:partition}  Domino tableaux that are split carry a particularly simple cycle structure.  A split domino tableau $T \in \SDT_r(n)$ consists of rows of horizontal and columns of vertical dominos.  The labels of the dominos in each such row and column form the complete set of its cycles.  Each is a core cycle and consequently open. Since $T$ has no non-core cycles, the set $\mathcal{C}(T)$ is in fact a left combinatorial $r$-cell.  Further, if we let $T' =\MT(T,\corecycles(T))$, then $T' \in \SDT_{r+1}(n)$ is also split and by Proposition \ref{proposition:MTDelta},
$\mathcal{C}(T) = \mathcal{C}(T')$.  This process can be repeated with $T'$ and iterated, eventually showing that:

\begin{proposition}\label{proposition:split}  $\mathcal{C}(T)$ is a left asymptotic cell whenever $T \in \SDT_r(n)$ is split.
\end{proposition}


For any $w \in W_n$ and $r \geq n-1$, the tableau pair $G_r(w)$ is always split.  Let $s(w)$ be the smallest integer such that $G_{s(w)}(w)$ is split and define $W_n^k$ to be the set of $w$ for which $s(w) = k$. Then:
$$W_n = \bigsqcup_{k=0}^{n-1} W_n^k.$$

\begin{definition}  We will say that $w$ is {\it non-split} if  $w \in W_n^{n-1}$ and {\it split} otherwise.  Extending this notion slightly, we will say $w \in W_n^k$ is {\it $k$-split} for $k < n-1$.
\end{definition}

The set of non-split elements in $W_n$  admits a simple description as detailed in \cite{howse:vogan}.  We reproduce it as part of the following proposition, whose proof follows directly from the definition of the map $G_{n-1}$.

\begin{proposition} Consider the signed permutation $w = w(1) w(2) \ldots w(n) \in W_n$ as a sequence, writing $x_1, \ldots, x_q$ for the subsequence of negative integers as well as $y_1, \ldots, y_{n-q}$ for the subsequence of positive integers contained therein.  We will say $w \in \area_n$ if and only if $\{|x_i|\}_i $and $\{y_i\}_i$ form decreasing sequences. Then $$W_n^{n-1} = \area_n.$$
\end{proposition}

\subsection{Unions of asymptotic cells}
We will now show that the  Kazhdan-Lusztig  and combinatorial left intermediate cells coincide.  The proof is different for split and non-split elements of $W_n$.  We summarize the salient results on Kazhdan-Lusztig cells in the following theorem, which is an amalgam of
Corollaries~5.9 and 8.11; Lemmas~5.10(iii), 8.3, and 8.7; and Theorem~8.8 of \cite{howse:vogan}:

\begin{theorem}[\cite{howse:vogan}]\label{theorem:howse} The sets of split and non-split elements of $W_n$ are both unions of left intermediate cells. Further:
    \begin{enumerate}
        \item[(a)] If $w$ is split, then its left intermediate cell is also a left asymptotic cell.
        \item[(b)] If $w$ is not split, then its left intermediate cell is the union of two left asymptotic cells and coincides with the set of all non-split elements with the same $\tau$-invariant as $w$.
    \end{enumerate}
\end{theorem}


%
%
%
%
%

We can show that the combinatorial left intermediate cells for split $w$ are simply left asymptotic cells.  In fact, a little more can be said.  The following is inspired by Example \ref{example:nonarea}:

\begin{proposition}\label{proposition:Wk} Suppose that $w \in W_n^k$. Then for all $r \geq k$, $\mathcal{C}_r(w)= \mathcal{C}_a(w)$.  In particular, if $w$ is split, then $\mathcal{C}_{n-2}(w) =  \mathcal{C}_a(w).$
\end{proposition}
\begin{proof}
If $r \geq k$, the tableaux pair $(S,T) : = G_r(w)$ is split.  Since split tableaux have no non-core cycles, $\mathcal{C}(T)$ is a combinatorial left $r$-cell and equals $\mathcal{C}_r(w)$.  By Proposition \ref{proposition:split}, $\mathcal{C}(T)$ is a left asymptotic cell and so it equals $\mathcal{C}_a(w)$.
\end{proof}

\begin{lemma}\label{lemma:notarea}
If $w$ is split, then $\mathcal{C}_{n-2}(w) = \mathcal{K}_{n-2}(w)$.
\end{lemma}

\begin{proof} This is a direct consequence of Proposition \ref{proposition:Wk} which gives us the equality $\mathcal{C}_{n-2}(w) = \mathcal{C}_a(w)$, and Theorem~\ref{theorem:howse}(a)
gives us $\mathcal{K}_{n-2}(w) = \mathcal{C}_a(w)$.
\end{proof}

Immediately, we obtain the following analogue of the first statement of Theorem~\ref{theorem:howse} for combinatorial cells. It also admits a straightforward independent proof.

\begin{proposition}\label{proposition:unions}
The sets of split and non-split elements of $W_n$ are both unions of combinatorial left intermediate cells.
\end{proposition}



The proof of  Lemma \ref{lemma:notarea} is straightforward partially because the extended cycles in $G_r(w)$ for $k$-split $w$ take on a particularly simple form whenever $r \geq k$. In general, if $r < k$ this cycle structure can be quite complex, see Section 2 of \cite{pietraho:equivalence}.  There is one instance when it is still tractable, mainly when $k = r+1$.  Analyzing this case will let us show that on Kazhdan-Lusztig and combinatorial left cells agree for non-split elements.  The proof proceeds via a sequence of propositions. The first details the cycle structure of the tableau pair $G_r(w)$ for $(r+1)$-split signed permutations and is modeled after Example \ref{example:area}.


\begin{proposition} \label{proposition:cycles2}
Suppose $w \in W_n^{r+1}$, $T:=Q_r(w)$, and $T' := Q_{r+1}(w)$.  Then:
\begin{enumerate}
    \item[(a)] There is a unique non-core open cycle $c$ in $T$.  With this exception, every other cycle in  $T$ is a core open cycle and consists of the labels of all horizontal dominos in a row or all vertical dominos in a column of $T$.
    \item[(b)]  With cycle $c$ as above, define tableaux $\overline{T} = \MT(T, c)$ and $\overline{T'} = \MTop(T', c)$.  Then  $\mathcal{C}(T) \cup \mathcal{C}(\overline{T})$ is a combinatorial left $r$-cell, and $$\mathcal{C}(T) \cup \mathcal{C}(\overline{T}) = \mathcal{C}(T') \cup \mathcal{C}(\overline{T'}).$$  When $c$ is a cycle consisting of the label of a single domino, then  $\mathcal{C}(T')$ as well as  $\mathcal{C}(\overline{T'})$ are left asymptotic cells and the above equation is a decomposition of a combinatorial left $r$-cell into two left asymptotic cells.
\end{enumerate}
\end{proposition}

\begin{proof}

As sets of labels, cycles in  $T$ coincide with  opposite cycles in $T'$, a split tableau.  The non-core cycles of $T$ are the non-core opposite cycles in $T'$.  There are exactly $r+2$ opposite cycles in $T'$, all open, and exactly one core opposite open cycle for each square in $\delta_{r+1}(T')$.  This leaves one opposite non-core open cycle $c$ in $T'$, proving part (a).

The first part of (b) is clear from the definition of combinatorial left cells since $c$ is the unique non-core open cycle in $T$.  The second part is just a restatement of Proposition \ref{proposition:MTDelta} adapted to the present setting, keeping in mind that here $\MT(T,\corecycles(T))$ may be either $T'$ or $\overline{T'}$.
For the last part, we have to show that when $c$ consists of the label of a single domino, then $\noncorecycles(T') = \noncorecycles(\overline{T'}) = \varnothing$.  For the tableau $T'$ this is straightforward: $T'$ is split and Proposition \ref{proposition:Wk} applies.  This is generally not true for $\overline{T'}$ unless $c$ is a singleton.  So assume that $c=\{m\}$ and that $D(m,T') =\{s_{ij},s_{i+1,j}\}$ is vertical.  Then $s_{i,j+1}$ is an empty square in $\delta_{r+3}(T')$ and since $c$ is a non-core cycle, $D(m, \overline{T'}) = \{s_{ij},s_{i,j+1}\}$ and $s_{i+1,j}$ is an empty square in $\delta_{r+3}(\overline{T'})$.
In particular, this means $\overline{T'}$ is split and $\noncorecycles(\overline{T'}) = \varnothing$.  If on the other hand we assume that $D(m,T')$ is horizontal, a similar argument applies. This gives us the desired decomposition of a combinatorial left $r$-cell in terms of two left asymptotic cells.
\end{proof}

\begin{proposition}\label{proposition:n-1}
Suppose $w$ and  $y$  are not split and let $S := Q_{n-1}(w)$, and $T :=Q_{n-1}(y)$.
The set $c= \{n\}$ is an opposite non-core open cycle in both $S$ and $T$.  If $\tau(S) = \tau(T)$, then either
 $T = S$ or  $T = \MTop(S, c).$
\end{proposition}

\begin{proof}
For the tableaux $S$ and $T$, the label of each of its $n$ dominos forms both a cycle and an opposite cycle.  Only $n-1$ of these are opposite core cycles, one for every square in on the diagonal $\delta_{n-1}$.  Because both tableaux are standard, the sole opposite non-core cycle $c$ must consist of the largest label which in this case is $n$.  Let $S' = \MTop(S,c)$.   By Lemma~\ref{lemma:tau:cycle}, $\tau(S) = \tau(S')$.  By Proposition~\ref{proposition:cycles2}(b), $\mathcal{C}(S)$ and $\mathcal{C}(S')$ are left asymptotic cells.  Further, it is easy to see that  $\mathcal{C}(S)$ and $\mathcal{C}(S')$ consist of non-split elements.  Similarly, we see that all elements of $\mathcal{C}(T)$ are not split and that it forms a left asymptotic cell.

By Theorem~\ref{theorem:howse}(b), there are exactly two left asymptotic cells within $W_n^{n-1}$ whose elements share the same $\tau$-invariant $\tau(S)$.   Since $\tau(S) = \tau(S') = \tau(T)$ and $S \neq S'$, the proposition follows.
\end{proof}

\begin{proposition}\label{proposition:n-2}
Suppose that $w$ and $y$ are not split.  If $\tau(w) = \tau(y)$, then $\mathcal{C}_{n-2}(w) =  \mathcal{C}_{n-2}(y)$.
\end{proposition}
\begin{proof}
Let $(S,T)  := G_{n-2}(w)$,  $(\overline{S},\overline{T}) := G_{n-2}(y)$,
$(S',T') := G_{n-1}(w)$, and $(\overline{S'},\overline{T'}) := G_{n-1}(y)$.
We need to show that $T = \MT(\overline{T}, U)$ for some, perhaps empty, set $U$ of non-core open cycles in $\overline{T}$.
Recall that $G_{n-2} = \Upsilon_{n-2} \circ G_{n-1}$.  From the definition of  $\Upsilon_{n-2}$ we obtain $T= \MTop(T', \oppcorecycles(T') \cup V_1)$ and $\overline{T}=\MTop(\overline{T'},\oppcorecycles(\overline{T'}) \cup V_2)$ for some subsets $V_1 \subset \oppnoncorecycles(T')$ and $V_2 \subset \oppnoncorecycles(\overline{T'})$.  By Proposition~\ref{proposition:n-1} we know that $V_1$ and $V_2$ are not complicated as $\oppnoncorecycles(T') = \oppnoncorecycles(\overline{T'}) = \{n\}$.

Since $\tau(w) = \tau(y)$, Corollary~\ref{proposition:tau} implies that $\tau(T') = \tau(\overline{T'})$ and again by  Proposition \ref{proposition:n-1}, we know that either $T'=\overline{T'}$ or $T' = \MTop(\overline{T'},\{n\})$.  But this means that either $T = \overline{T}$ or $T = \MT(\overline{T},\{n\})$, as desired.
\end{proof}

\begin{lemma}\label{lemma:area}
If $w$ is not split, then $\mathcal{C}_{n-2}(w) = \mathcal{K}_{n-2}(w)$.
\end{lemma}

\begin{proof}
Restricting our attention to the set of non-split elements of $W_n$,
    \begin{itemize}
        \item Proposition \ref{proposition:n-2} implies that the $\tau$-invariant defines a partition finer than the one given by combinatorial left intermediate cells,
        \item Corollary~\ref{proposition:tau} implies that the partition into combinatorial left $r$-cells is finer than the $\tau$-invariant partition for all $r \geq 0$, and
        \item  Theorem~\ref{theorem:howse}(b) implies that the $\tau$-invariant partition agrees with the one given by left intermediate cells.
    \end{itemize}
Taken together with Proposition~\ref{proposition:unions}, the lemma follows.
\end{proof}

From Lemmas \ref{lemma:notarea} and \ref{lemma:area} we can now obtain the following explicit formulation of Conjecture \ref{conjecture:bgil}:

\begin{theorem} Consider $w,y \in W_n$, write $r = n-2$, and let $$G_{r}(w) = (S_1,T_1) \text{ while } G_{r}(y) = (S_2,T_2).$$
Then $w \approx_r^L y$  if and only if $T_2 =\MT(T_1,U)$ for some $U \subset \noncorecycles(T_1)$.
\end{theorem}

In the intermediate parameter setting, Theorem~\ref{theorem:howse} shows that {\it all} left cells  either coincide with, or are unions of, asymptotic left cells.  Together with Conjecture~\ref{conjecture:bgil}, the combinatorics of tableaux in Propositions~\ref{proposition:Wk}, \ref{proposition:cycles2}, and \ref{proposition:MTDelta} suggest that this is still the case for specific Kazhdan-Lusztig left $r$-cells with $r < n-2$.

%
%
%


\end{document}